\documentclass[12pt,psamsfonts]{amsart}
\usepackage{amsmath}
\usepackage{amsthm}
\usepackage{amssymb}
\usepackage{amscd}
\usepackage{amsfonts}
\usepackage{amsbsy}

\def\be{\begin{equation}}
\def\ee{\end{equation}}

\def\bq{\begin{eqnarray}}
\def\eq{\end{eqnarray}}

\def\beq{\begin{eqnarray*}}
\def\eeq{\end{eqnarray*}}

\newtheorem {theorem} {Theorem}%[section]

\newtheorem {corollary} [theorem]{Corollary}

\newtheorem {conjecture} [theorem]{Conjecture}
\newcommand{\R}{\mathbb{R}}

\begin{document}
\title[Various aspects of differential equations  ]
{Various aspects of differential equations having a complete set of independent first integrals }
\author[ R. Ram\'{\i}rez and N. Sadovskaia]
{ Rafael Ram\'{\i}rez$^2$ and Natalia Sadovskaia$^3$}
\address{$^2$ Departament d'Enginyeria Inform\`{a}tica i Matem\`{a}tiques, Universitat Rovira i Virgili, Avinguda dels Pa\"{\i}sos Catalans 26, 43007 Tarragona, Catalonia, Spain.}
\email{rafaelorlando.ramirez@urv.cat}

\address{$^3$ Departament de Matem\`{a}tica Aplicada II, Universitat Polit\`{e}cnica de Catalunya, C. Pau Gargallo 5, 08028 Barcelona, Catalonia, Spain.}
\email{natalia.sadovskaia@upc.edu}
\date{}
\dedicatory{}
\maketitle

\maketitle

\begin{abstract}

In this paper we study the differential equations in $D\subseteq \R^{2N}$ having a complete set of independent first integrals. In particular we study the case when the first integrals are
\[f_\nu=(Ax_\nu+By_\nu)^2+\displaystyle\sum_{j=1}^{N}\dfrac{\left(x_\nu y_j-x_jy_\nu\right)^2}{a_\nu-a_j},\]for $\nu=1,\ldots,N,$
where $A,B$ and $a_1<a_2\ldots<a_N$ are constants.
\end{abstract}



\section{Introduction.}
Let $\textsc{D}$ be an open subset of $\mathbb{R}^{2N}.$ By
definition an autonomous differential system is a system of the
form
\begin{equation}\label{0}
\dot{\textbf{x}}=\textbf{X}(\textbf{x}),\quad
\textbf{x}\in\textsc{D},
\end{equation}
where the dependent variables $\textbf{x}=(x_1,\ldots,x_{2N})$ are
real, the independent variable (the time $t$) is real and the
$\mathcal{C}^1$ functions $\textbf{X}(\textbf{x})=
(X_1(\textbf{x}), ...,X_K(\textbf{x}))$ are defined in the open
set $\textsc{D}.$

\smallskip

For simplicity we shall assume the underlying functions to be of
 class $ \mathcal{C}^\infty,$ although most results remain valid
 under weaker hypotheses.

 Below we use the following notations
\[\begin{array}{rl}
 |S|=&\{h_1,\ldots,h_{2N}\}^*=\left|
\begin{array}{ccccc}
dh_1(\partial_1)& \ldots& dh_1 (\partial_{2N}) \\
\vdots & & \vdots \\
\vdots & & \vdots \\
dh_{2N}(\partial_1)
& \ldots & dh_{2N}(\partial_{2N})\\
\end{array}
\right|,\vspace{0.20cm}\\
x_{N+j}=&y_j,\vspace{0.20cm}\\
\end{array}
\]for $j=1,\ldots,N.$
Our main result is the following
\begin{theorem}\label{main5}
Let $f_j=f_j(x_1,\ldots,x_N,\,y_1,\ldots,y_N)$ for $j=1,2,\ldots,N$  be a
given set of independent functions defined  in an open set
$\textsc{D}\subset{\mathbb{R}^{2N}}$ and such that

i)\quad If
\[|S|_0= \{f_1,\ldots,f_{N},x_1,\ldots,x_N\}^*\ne 0.\]
 Then the differential systems in $\textsc{D}$ which admit the set of
first integrals $f_j$ for $j=1,2,\ldots,N$ are
\begin{equation}
\label{550}
\begin{array}{rl}
\dot{x}_k=&\dfrac{\partial H}{\partial y_k},\vspace{0.20cm}\\
 \dot{y}_k=&-\dfrac{\partial H}{\partial
x_k}+\vspace{0.20cm}\\
&\dfrac{1}{|S|_0}\displaystyle\sum_{j=1}^{N}\{H,f_j\}\{f_1,\ldots,f_{j-1},y_k,f_{j+1},\ldots,f_N,x_1,\ldots,x_N\}^*,
\end{array}
\end{equation} for $k=1,\ldots,N$ where
$H=H(x_1,\ldots,x_N,y_1,\ldots,y_N),$ is an arbitrary function and
\begin{equation}
\label{55}
\displaystyle\sum_{n=1}^{N}\left(\frac{\partial H}{\partial y_j}
\frac{\partial H}{\partial y_j}\frac{\partial f_\alpha}{\partial
x_j}-\frac{\partial H}{\partial y_j}\frac{\partial H}{\partial
y_j}\frac{\partial f_\alpha}{\partial
x_j}\right)=\{H,\,f_\alpha\}.
\end{equation}

ii)\quad If
\[|S|_0=0,\quad |S|_N= \{f_1,\ldots,f_{N},x_1,\ldots,x_{N-1},y_1\}^*\ne 0.\]
Then the differential systems in $\textsc{D}$ which admit the set of
first integrals $f_j$ for $j=1,2,\ldots,N$ are
\begin{equation}\label{5}
\begin{array}{rl}
 \dot{x}_k=&\dfrac{\partial H}{\partial y_k},\vspace{0.20cm}\\
 \dot{y}_k=&-\dfrac{\partial
H}{\partial
 x_k}+\vspace{0.20cm}\\
 &\dfrac{1}{|S|_N}\displaystyle\sum_{j=1}^{N}\{H,f_j\}\{f_1,\ldots,f_{j-1},y_k,f_{j+1},\ldots,f_N,x_1,\ldots,x_N\}^*+\vspace{0.20cm}\\
 &\lambda\,\{f_1,\ldots,f_{N},x_1,\ldots,x_{k-1},y_1,x_{k+1},\ldots,x_{N-1}\}^*,
\end{array}
 \end{equation} for $k=1,\ldots,N,$ where $H=H(x_1,\ldots,x_N,y_1,\ldots,y_N),$ and $\lambda=\lambda\,(x_1,\ldots,x_N,y_1,\ldots,y_N)$ are arbitrary
functions .
\end{theorem}
\section{Proof of Theorems \ref{main5}}
 \begin{proof}[Proof of Theorem \ref{main5}]
In view of the relations
\[\displaystyle\sum_{j,k=1}^{N}\{H,f_j\}\{f_1,\ldots,f_{j-1},y_k,f_{j+1},\ldots,f_N\}^*\,\dfrac{\partial f_\alpha}{\partial y_k}=\{H,\,f_\alpha\},\]
we deduce that $ \dot{f}_\alpha=\{H,\,f_\alpha\}-\{H,\,f_\alpha\}=0,$
for $\alpha=1,\ldots,N.$  Thus the given functions are
constant along the solutions of the given  differential system. This is
the proof of part i) of the theorem.

\smallskip

Now we prove the part ii). After computation we obtain
\[
  \dot{f}_\alpha=\lambda\{f_1,\ldots,f_N\}^*_0\frac{\partial\,f_\alpha}{\partial
  x_N},
\]for $\alpha=1,\ldots,N,$
hence in view of the assumptions we obtain the proof of part ii).  In short Theorem \ref{main5} is proved.
\end{proof}
We shall illustrate these results in the following example.
\bigskip

\noindent{\bf Example 1.} For the case when
\[ \begin{array}{rl}
f_1=H=&\dfrac{y^2_1}{2m_1}+\dfrac{y^2_2}{2m_2}+\dfrac{y^2_3}{2m_3}+\dfrac{a}{(x-y)^2}+\dfrac{b}{(z-x)^2}\dfrac{c}{(z-y)^2},\vspace{0.20cm}\\
f_2=&x_1y_1+x_2y_2+x_3y_3,\vspace{0.20cm}\\
f_3=&y_1+y_2+y_3,
\end{array}
\]where $a,b,c$ are constants, we obtain that
\[|S|_0=\dfrac{(x_2-x_1)y_3}{m_3}+\dfrac{(x_1-x_3)y_2}{m_2}+\dfrac{(x_3-x_2)y_1}{m_1}.\]
By considering that \[\{f_1,f_2\}=2f_1,\quad\{f_3,f_2\}=-f_3,\quad \{f_1,f_3\}=0,\] we deduce that differential system \eqref{550}
 takes the form
  \[ \begin{array}{rl}
   \dot{x}_1=&\dfrac{\partial H}{\partial y_1},\quad\dot{x}_2=\dfrac{\partial H}{\partial y_2},\quad\dot{x}_3=\dfrac{\partial H}{\partial y_3},\vspace{0.20cm}\\
 \dot{y}_1=&-\dfrac{\partial
H}{\partial
 x_1}+2H\left(\dfrac{y_3}{m_3}-\dfrac{y_2}{m_2}\right),\vspace{0.20cm}\\
  \dot{y}_2=&-\dfrac{\partial
H}{\partial
 x_2}+2H\left(\dfrac{y_1}{m_1}-\dfrac{y_3}{m_3},\right),\vspace{0.20cm}\\

 \dot{y}_3=&-\dfrac{\partial
H}{\partial
 x_3}+2H\left(\dfrac{y_2}{m_2}-\dfrac{y_1}{m_1}\right).\vspace{0.20cm}\\
  \end{array}
\]

We observe that these first integrals appear when we  study the movement
of three particles with masses $m_1,m_2,m_3,$ which interact with each other with a force inversely proportional
 to the cube of the distance between them.

\noindent{\bf Example 2.} It is well known that in the Kepler
problem there are six first integrals
\[\textbf{M}=(M_1,M_2,M_3)=\textbf{x}\times
\textbf{y}=\textbf{c}_1,\quad
\textbf{W}=(W_1,\,W_2,\,W_3)=\textbf{y}\times
\textbf{M}+\dfrac{\mu \textbf{x}}{r}=\textbf{c}_2,\] where
$r=||\textbf{x}||=\sqrt{x^2_1+x^2_2+x^2_3
}.$

We determine differential system with three first integrals
$(M_1,M_2,M_3)$ and $(W_1,\,W_2,\,W_3).$ For the first case we
obtain that
\[ \begin{array}{rl}
f_1=M_1,\,f_2=M_2,\,f_3=M_3,\\
 |S|_1=x_1M_3,\quad
|S|_0=0,\\
\{f_1,f_2\}´=f_3,\quad \{f_2,f_3\}´=f_1,\quad\{f_3,f_1\}´=f_2.
\end{array}
\] Choosing $H=\dfrac{1}{2}\left(y^2_1+y^2_2+y^2_3\right)$ then system \eqref{5} takes the form
\[\dot{\textbf{x}}=\textbf{y},
\quad \dot{\textbf{y}}=\lambda\,\textbf{x}.
 \]
  These equations can be
interpreted as geodesic flow of a particle which is constrained to
move on the unit sphere $x^2_1+x^2_2+x^2_3=1.$

For the second case we obtain that
\[\begin{array}{rl}
f_1=W_1,\,f_2=W_2,\,f_3=W_3,\\
|S|_0=2\left<\textbf{x},\textbf{y}\right>||\textbf{M}||^2\ne
0,\end{array}\]where $\left<\textbf{x},\textbf{y}\right>=x_1y_1+x_2y_2+x_3y_3,$ \,$||\textbf{M}||^2=M^2_1+M^2_2+M^2_3,$  Taking $H=\dfrac{1}{2}\left(y^2_1+y^2_2+y^2_3\right)$ then from  \eqref{550} we obtain the differential system
\[\dot{\textbf{x}}=\textbf{y},\quad \dot{\textbf{y}}=-\dfrac{\textbf{x}}{||\textbf{x}||^3}.
\]

It is easy to observe that the first integrals $W_1,\,W_2,\,W_3$ admit the
representation
\[ W_j=\dfrac{\partial F}{\partial x_j}\] for $j=1,2,3$ where
\[ F=1/2\left(||\textbf{x}||^2||\textbf{y}||^2-\left<\textbf{x},\,\textbf{y}\right>+\mu
||\textbf{x}||\right).\]

\bigskip

\noindent{\bf Example 3.} We determine the differential system
\eqref{5} in the case
\[\begin{array}{rl}
f_1=\Upsilon_1x_1+\Upsilon_2x_2+\Upsilon_3x_3,\quad
f_2=\Upsilon_1y_1+\Upsilon_2y_2+\Upsilon_3y_3,\\
 f_3=\dfrac{1}{2}\left(
\Upsilon_1(x^2_1+y^2_1)+\Upsilon_2(x^2_2+y^2_2)+\Upsilon_3(x^2_3+y^2_3)\right),
\end{array}\] where $\Upsilon_j=constants$ for $j=1,2,3.$
In this case we have
\[\begin{array}{rl}
|S|_0=&0,\quad |S|_1=a_2a^2_3(y_3-y_2),\vspace{0.20cm}\\
\{f_2,f_1\}=&\Upsilon_1+\Upsilon_2+\Upsilon_3,\quad \{f_3,f_1\}= f_2 ,\quad \{f_3,f_2\}=- f_1
\end{array}\]
Clearly that $|S|_1\ne 0$ if $y_2-y_3\ne 0.$

Taking \[\lambda_4=\dfrac{1}{\Upsilon_1}\dfrac{\partial
H}{\partial y_1},\quad
\lambda_5=\dfrac{1}{\Upsilon_2}\dfrac{\partial H}{\partial y_2},\]
we obtain that \eqref{5} takes the form
\[
\begin{array}{rl}
\Upsilon_1\dot{x}_1=&\dfrac{\partial H}{\partial y_1},\quad \Upsilon_1\dot{x}_2=\dfrac{\partial H}{\partial y_2},\quad \Upsilon_1\dot{x}_3=\dfrac{\partial H}{\partial y_3},\vspace{0.20cm}\\
\Upsilon_1\dot{y}_1=&-\dfrac{\partial H}{\partial
x_1}+\lambda{(y_2-y_3)},\vspace{0.20cm}\\
\Upsilon_2\dot{y}_2=&-\dfrac{\partial
H}{\partial x_2}+\lambda(y_3-y_1),\vspace{0.20cm}\\
\Upsilon_3\dot{y}_3=&-\dfrac{\partial H}{\partial
x_3}+\lambda(y_1-y_2),
\end{array}
\]where $\lambda$ is an arbitrary function and $H$ is the Hamiltonian function such that
\[\displaystyle\sum_{j=1}^3\frac{1}{\Upsilon_j}\left(\dfrac{\partial H}{\partial
y_j}\dfrac{\partial f\alpha}{\partial x_j}-\dfrac{\partial
H}{\partial x_j}\dfrac{\partial f\alpha}{\partial y_j}\right)=0,\]
for $\alpha=1,2,3.$

In particular if
\[H=\displaystyle\sum_{\substack{m,k=1\\m\ne{k}}}^{3}\Upsilon_m\Upsilon_k\lg\sqrt{(x_k-x_m)^2+(y_k-y_m)^2},\]
and taking $\lambda=0$ we obtain Hamiltonian differential
equations of motion of three vortices with intensities
$\Upsilon_j$ for $j=1,2,3,$ (see for instance \cite{Koz}).

\bigskip

\noindent{\bf Example 4.} Now we study the following case

\[f_\nu=(Ax_\nu+By_\nu)^2+\displaystyle\sum_{j\ne \nu}^N\frac{\left(x_\nu
y_j-x_jy\nu\right)^2}{a_\nu-a_j},\] for $\nu=1,\ldots,N,$ where
$A$ and $B$ are constants.

It is easy to show that in all the cases the first integrals are
in involution . Thus if we determine $H=H(f_1,\ldots,f_N)$ then we obtain
the completely integrable Hamiltonian system.

After some computation we obtain that $|S|_0\ne 0$
if $B\ne 0.$

We apply Theorem \ref{main5} for these cases when $N=3.$

For the case when $B=0$  we obtain, after some calculations that
$|S|_0=0$ and
\[|S|_1=\dfrac{K}{\Delta}x_3x_1,\quad
|S|_2==\dfrac{K}{\Delta}x_3x_2,\quad
|S|_3=\dfrac{K}{\Delta}x_3x_3,\] where
$\Delta=(a_1-a_2)(a_2-a_3)(a_1-a_3),$ and
\[\begin{array}{rl}
K=&a_1(x_2y_3-x_3y_2)\left(x_2(x_1y_2-x_2y_1)-x_1(x_3y_1-x_1y_3)\right)+\vspace{0.20cm}\\
&a_2(x_3y_1-x_1y_3)\left(x_3(x_2y_3-x_3y_2)-x_1(x_1y_2-x_2y_1)\right)+\vspace{0.20cm}\\
&a_3(x_1y_2-x_2y_3)\left(x_1(x_3y_1-x_1y_3)-x_2(x_2y_3-x_3y_2)\right)\end{array}\]
thus the differential system \eqref{5} takes the form
\[
\begin{array}{rl}
\dot{x}_k=&\dfrac{\partial H}{\partial y_k},\vspace{0.20cm}\\
\dot{y}_k=&-\dfrac{\partial H}{\partial x_k}+\lambda\,x_k.
\end{array}
\]These differential equations described the behavior of the
particle with Hamiltonian $H$ and constrained to move on the
sphere $x^2_1+x^2_2+x^2_3=1.$

 In particular if we take
\[
\begin{array}{rl}
H=&\dfrac{1}{2}\left(a_1f_1+a_2f_2+a_3f_3\right)=
1/2\left(||\textbf{x}||^2||\textbf{y}||^2-\left<\textbf{x},\,\textbf{y}\right>+a_1x^2_1+a_2x^2_2+a_3x^2_3\right)\\
\lambda=&\Psi (x^2_1+x^2_1+x^2_1),
\end{array}
\]  then from the above equations
we deduce the equation of motion of a particle on an 3-dimensional
sphere, with an anisotropic harmonic potential. This system is one
of the best understood integrable systems of classical mechanics
(for more details see \cite{Moser}).

For the case when $B\ne 0$  we obtain that $f_1,f_2,f_3\}_0\ne{0}.$

\section*{Acknowledgements}

The first author was partly supported by the Spanish Ministry of
Education through projects TSI2007-65406-C03-01 "E-AEGIS" and
Consolider CSD2007-00004 "ARES".

\end{document}